\documentclass[11pt]{amsart}
\usepackage{amsmath,amsthm, amsfonts, amssymb, latexsym,verbatim,bbm,longtable}
\input xy
\xyoption{all}
\usepackage{hyperref,multirow}
\usepackage[shortlabels]{enumitem}
\usepackage{young}

\allowdisplaybreaks[2] \textwidth15.1cm \textheight22cm \headheight12pt \oddsidemargin.4cm
\theoremstyle{plain}
\newtheorem{theorem}{Theorem}[section]
\newtheorem{thm}[theorem]{Theorem}
\newtheorem{lemma}[theorem]{Lemma}

\newtheorem{prop}[theorem]{Proposition}
\newtheorem{example}[theorem]{Example}

\newtheorem{defn}[theorem]{Definition}

\newtheorem{cor}[theorem]{Corollary}
\newtheorem{conjecture}[theorem]{Conjecture}

\newtheorem{rmk}[theorem]{Remark}

\newcommand{\iso}{\cong}
\newcommand{\arr}{\rightarrow}

\newcommand{\incl}{\hookrightarrow}

\DeclareMathOperator{\Spec}{Spec}

\newcommand{\C}{\mathbb{C}}

\newcommand{\Q}{\mathbb{Q}}

\DeclareMathOperator{\SL}{SL}

\begin{document}

\title[Billey-Postnikov decompositions and fibre bundles]{Billey-Postnikov
decompositions and the fibre bundle structure of Schubert varieties}

\author{Edward Richmond}
\email{edward.richmond@okstate.edu}

\author{William Slofstra}
\email{weslofst@uwaterloo.ca}

\begin{abstract}
    A theorem of Ryan and Wolper states that a type $A$ Schubert variety is
    smooth if and only if it is an iterated fibre bundle of Grassmannians. We
    extend this theorem to arbitrary finite type, showing that a Schubert
    variety in a generalized flag variety is rationally smooth if and only if
    it is an iterated fibre bundle of rationally smooth Grassmannian Schubert
    varieties. The proof depends on deep combinatorial results of
    Billey-Postnikov on Weyl groups. We determine all smooth and
    rationally smooth Grassmannian Schubert varieties, and give a new proof of
    Peterson's theorem that all simply-laced rationally smooth Schubert
    varieties are smooth. Taken together, our results give a fairly complete
    geometric description of smooth and rationally smooth Schubert varieties
    using primarily combinatorial methods.
\end{abstract}
\maketitle

\section{Introduction}

Let $G$ be a connected semisimple algebraic group over an algebraically closed field
$k$, and let $X$ be a partial flag variety of type $G$. The variety $X$ is
stratified by Schubert varieties $X(w)$, which are indexed by minimal length
coset representatives $w$ in the Weyl group $W$ of $G$. When $k = \C$ and $G =
\SL_n(\C)$ is the semisimple group of type $A_{n-1}$, the partial flag
varieties parametrize flags of subspaces in $\C^n$. When $X$ is the
complete flag variety in type $A$, Ryan proved that smooth Schubert varieties
are iterated fibre bundles of Grassmannians \cite{Ry87}.  This result was
extended to the partial flag varieties of type $A$ over any algebraically
closed field $k$ of characteristic zero by Wolper \cite{Wo89}.  The main result
of this paper is an extension of Ryan and Wolper's theorems to flag varieties
of all finite types.  We consider both the class of smooth Schubert varieties,
and the larger class of rationally smooth Schubert varieties, proving the
following result: a finite type Schubert variety is (rationally) smooth if and
only if it is an iterated fibre bundle of (rationally) smooth Grassmannian
Schubert varieties.

In type $A$, every smooth Grassmannian Schubert variety is a sub-Grassmannian
\cite[Theorem 9.3.1]{BL00}, and thus we recover Ryan and Wolper's theorem from
our results. However, in other types there are (rationally) smooth Grassmannian
Schubert varieties which are not sub-Grassmannians.  Fortunately, the singular
locus of these Schubert varieties has been extensively studied \cite{LW90}
\cite{BP99}; a summary can be found in \cite{BL00}. More recently, smooth
Grassmannian Schubert varieties have been studied in the context of homological
rigidity \cite{Ro12} \cite{HM13}. From this work, the list of smooth Schubert
varieties is known for many generalized Grassmannians. As part of our results,
we finish this line of inquiry by giving a complete list of both smooth and
rationally smooth Grassmannian Schubert varieties in all finite types.
When combined with our extension of Ryan and Wolper's theorem, this gives
a fairly complete geometric description of (rationally) smooth Schubert
varieties in any finite type.

There are several well known characterizations of (rationally) smooth Schubert varieties.  If $P_w(t)$ is the Poincar\'{e} polynomial of $X(w)$, then a theorem of Carrell-Peterson states that $X(w)$ is rationally smooth if and only
if $P_w(t)$ is a palindromic polynomial \cite{Ca94}.  Another characterization
says that $X(w)$ is rationally smooth if and only if the Weyl group element $w$
has trivial Kazhdan-Lusztig polynomials.  For type $A$ Schubert
varieties, the Lakshmibai-Sandya theorem \cite{LS90} states that a Schubert
variety $X(w)$ is smooth if and only if $w$, as a permutation, avoids $3412$
and $4231$. This pattern-avoidance criterion has been extended to classical
types by Billey \cite{Bi98} and to all finite types using root-system pattern
avoidance by Billey-Postnikov \cite{BP05}. These pattern avoidance criteria
provide an efficient way to test if a given Schubert variety $X(w)$ is smooth
or rationally smooth. In contrast, our results provide information about the
geometric structure of (rationally) smooth Schubert varieties. For instance,
Peterson's theorem, proved in full generality by Carrell-Kuttler, states that
if $G$ is simply-laced then $X(w)$ is rationally smooth if and only if it is
smooth \cite{CK03}. Other proofs of Peterson's theorem have been given in
\cite{Dye} and \cite{JW12}. Using our results, we give a new proof of
Peterson's theorem; compared to the proofs listed above, our proof is more
combinatorial. Using the list of (rationally) smooth Grassmannian Schubert
varieties, we can also efficiently generate the list of (rationally) smooth
Schubert varieties in any partial flag variety. This method can be extended to
enumerate (rationally) smooth Schubert varieties in complete flag varieties;
since this requires a new combinatorial data structure, we leave this to
another paper.

The central idea behind our results is the notion of a Billey-Postnikov
decomposition. To explain this, let $X^J$ denote the partial flag variety
associated to a subset $J$ of the simple generators $S$ of $W$. If $J \subseteq
K \subseteq S$, then there is a natural projection $\pi : X^J \arr X^K$.
If $X^J(w)$ is a Schubert variety of $X^J$ and $w = vu$ is the
parabolic decomposition of $w$ with respect to $K$, then the restriction of
$\pi$ to $X^J(w)$ gives a projection
\begin{equation}\label{E:Schubert_Proj}
    \pi:X^J(w)\arr X^K(v)
\end{equation}
with generic fibre $X^J(u)$. While $\pi : X^J \arr X^K$ is a fibre bundle, the
restriction in equation \eqref{E:Schubert_Proj} is not a fibre bundle in
general (see, for instance, example \ref{Ex:forintro}). If this projection is a
fibre bundle, then the Leray-Hirsch theorem states that the singular cohomology
ring $H^*(X^J(w))$ is a free $H^*(X^K(v))$-module over $H^*(X^J(u))$. As a
consequence, if $P_w^J(t)$ denotes the Poincar\'{e} polynomial of $X^J(w)$,
then
\begin{equation}\label{E:PolyFactor}
    P_w^J(t)=P_v^K(t)\cdot P_u^J(t).
\end{equation}
We prove that equation \eqref{E:PolyFactor} is also a sufficient condition for
the projection in equation \eqref{E:Schubert_Proj} to be a fibre bundle.  This
result is stated in Theorem \ref{T:bp1}, and holds for Schubert varieties of
any Kac-Moody group.

Factorizations of $P_w^J(t)$ have been studied by a number of authors, most
notably by Gasharov \cite{Ga98}, Billey \cite{Bi98}, and Billey-Postnikov
\cite{BP05}. Billey and Postnikov show that if $X^{\emptyset}(w)$ is a
rationally smooth Schubert variety of finite type, then there exists a subset
$K$ of simple generators for which either $w$ or $w^{-1}$ has a parabolic
decomposition $vu$ with respect to $K$ such that the Poincar\'{e} polynomial
factors as $P_w^{\emptyset}(t)=P_v^K(t)\cdot P_u^{\emptyset}(t)$ \cite{BP05}.
Moreover, $K$ can be choosen such that $S\setminus K=\{s\}$ for some leaf $s$
of the Dynkin diagram of $G$. In \cite{OY10}, Oh and Yoo call such a parabolic
decomposition a Billey-Postnikov decomposition.  In this paper, we say that a
parabolic decomposition $w=vu$ with respect to $K$ is \emph{Billey-Postnikov
(BP)} if the polynomial $P_w^J(t)$ factors as in equation \eqref{E:PolyFactor}
(dropping the condition that $S\setminus K=\{s\}$ for some leaf $s$). When
$J = \emptyset$, this agrees with the definition used by the authors in
\cite{RS12}. We then prove that if $X^J(w)$ is rationally smooth and
of finite type, then there exists a nontrivial $K$ containing $J$ for which $w$
has a BP decomposition. This result is stated in Theorem \ref{T:bp2}. The proof
uses the existence theorem in \cite{BP05} combined with an inductive argument.
Note that we cannot apply the existence theorem for BP decompositions stated in
\cite{BP05} to construct fibre bundle structures on Schubert varieties
directly. While the Bruhat intervals $[e,w]$ and $[e,w^{-1}]$ are order
isomorphic, the Schubert varieties $X^{\emptyset}(w)$ and
$X^{\emptyset}(w^{-1})$ are not necessarily isomorphic.  Hence a fibre bundle
structure on $X^{\emptyset}(w^{-1})$ may not yield a fibre bundle structure on
$X^{\emptyset}(w)$.

\subsection{Acknowledgements}
We would like to thank Dave Anderson, Sara Billey, Jim Carrell, and Alex Woo
for helpful discussions. We thank the anonymous referee for useful suggestions
on the manuscript. The first author was partially supported by the Natural
Sciences and Engineering Research Council of Canada.

\section{Background and terminology}

We use the notation from the introduction throughout the paper. In particular,
we work over a fixed algebraically closed field $k$ of arbitrary
characteristic. Let $G$ denote a semisimple algebraic group over $k$ or, as long as
$k = \C$, a Kac-Moody group. When working with Kac-Moody groups, we take $G$
to be the minimum Kac-Moody group $G^{\min}$ as defined in \cite{Ku02}.
Fix a choice of maximal torus and Borel $T \subset B \subset G$.
Let $W=N(T)/T$ denote the Weyl group of $G$ and fix a simple generating set $S$
for $W$. We choose a representative in the normalizer $N(T)$ of $T$ for each
element $w\in W$. If $S=\{s_1,s_2,\ldots, s_n\}$ is the generating set of a
specific Weyl group of finite type, we use the Bourbaki labeling of Dynkin
diagrams in \cite{Bo68} to describe the Coxeter relations.

For basic facts about Schubert varieties, we refer again to \cite{Ku02}, and
in particular to the facts about Tits systems in \cite[section 5.1]{Ku02}.  As
in loc. cit., a \emph{(standard) parabolic subgroup} of $W$ is a subgroup
$W_J$ generated by a subset $J \subset S$. If $J \subseteq S$, we let $P_J := B
W_J B$ be the \emph{parabolic subgroup} of $G$ corresponding to $J$. The
partial flag variety associated to $J$ is then $X^J := G / P_J$.  Let $W^J \simeq W / W_J$ denote the set of minimal length
coset representatives.  If $w \in W^J$, then the
\emph{Schubert variety} associated to $w$ is the closure $X^J(w):=\overline{B w
P_J}/P_J$.  The dimension of $X^J(w)$ is $\ell(w)$, the Coxeter length of
$w$ with respect to $S$.  Note that for Kac-Moody groups $G$, the flag variety $X^J$ can be an
infinite-dimensional ind-variety, but the Schubert varieties $X^J(w)$ are
always finite-dimensional.

If $J \subseteq K \subseteq S$, then every element $w \in W^J$ can be written
uniquely as $w = vu$, where $v \in W^K$ and $u \in W_K \cap W^J$. This is
called the \emph{(right) parabolic decomposition} of $w$ with respect to $K$.
We say $w=w_1\cdots w_k$ is a \emph{reduced decomposition} if
$\ell(w)=\sum_i \ell(w_i)$.  Note that all parabolic decompositions are
reduced.

If $\leq$ denotes Bruhat order on $W$, then
\begin{equation*}
    X^J(w) = \coprod_{x \in [e,w] \cap W^J} B x P_J/P_J
\end{equation*}
where $[e,w]$ denotes the interval in Bruhat order between $e$ and $w$.  If
$W_J$ is finite and $w \in W^J$, then the preimage of $X^J(w)$ in $G/B$ is
$X^{\emptyset}(w')$, where $w'$ is the maximal length representative of the
coset $w W_J$.  Alternatively, $w' = wu_0$, where $u_0$ is the longest element
of $W_J$. If $w_1$ and $w_2$ are in $W^J$, then $w_1 \leq w_2$ if and only if
$w_1' \leq w_2'$, where $w_i'$ is the longest element in the coset $w_i W_J$.
Geometrically, $w_1 \leq w_2$ if and only if $X^J(w_1) \subseteq X^J(w_2)$.

The Poincar\'{e} polynomial $P_w^J(t)$ of an element $w \in W^J$ is defined as
\begin{equation*}
    P_w^J(t) := \sum_{x \in [e,w] \cap W^J} t^{\ell(x)}.
\end{equation*}
It is well-known that $P_w^J(t^2) = \sum \dim H^i(X^J(w))\, t^i$, the
Poincar\'{e} polynomial for the cohomology of $X^J(w)$. The Poincar\'{e}
polynomial is \emph{palindromic} if $P_w^J(t) = t^{\ell(w)} P_w^J(t^{-1})$,
or in other words if the coefficients form a palindrome when ordered by degree.

The support of an element $w \in W$ is defined as
\begin{equation*}
    S(w):=\{s\in S \ |\ s\leq w\}.
\end{equation*}
Equivalently, $S(w)$ is the set of simple reflections appearing in any reduced
decomposition of $w$. The left and right descents set of $w$ are
\begin{align*}
    D_L(w)&:=\{s\in S \ |\ \ell(sw)\leq \ell(w)\} \\
    D_R(w)&:=\{s\in S \ |\ \ell(ws)\leq \ell(w)\}.
\end{align*}
A \emph{generalized Grassmannian} is a flag variety $X^J$ where $|S\setminus
J|=1$. In other words, $P_J$ is a maximal parabolic subgroup of $G$ where $X^J=G/P_J$. A Schubert
variety $X^J(w)$ of a generalized Grassmannian is called a \emph{Grassmannian
Schubert variety}.  An element $w\in W$ is called a \emph{Grassmannian element}
if $w\in W^J$ for some generalized Grassmannian $X^J$. Equivalently, $w \in
W$ is Grassmannian if and only if $w$ has a unique right descent. By a
\emph{Grassmannian parabolic decomposition}, we mean a parabolic decomposition
$w = vu$ with respect to a set $K$ such that $|K \cap S(w)| = |S(w)| - 1$.

For more information about Schubert varieties over an arbitrary field, we point
to \cite{Bo91} and \cite{BK04}.

%\subsection{Organization}
%The main results are stated in Section \ref{S:main}. Section \ref{S:bpchar} is
%concerned with characterizations of Billey-Postnikov decompositions, including
%our geometric characterization. In Section \ref{S:bpexist}, we restate the results of
%Billey-Postnikov and others on the existence of Billey-Postnikov
%decompositions, and classify rationally smooth Grassmannian Schubert varieties.
%In Section \ref{S:onesided} we prove our main theorem on the existence of
%Billey-Postnikov decompositions. In Section \ref{S:applications} we
%finish the proof of the Ryan-Wolper theorem.

\section{Main results}\label{S:main}

As stated in the introduction, we define Billey-Postnikov decompositions as follows:
\begin{defn}\label{D:BP}
    Let $w \in W^J$, and let $w = vu$ be a parabolic decomposition with respect
    to $K$, where $J \subseteq K \subseteq S$. We say that $w = vu$ is a
    \emph{Billey-Postnikov (BP) decomposition with respect to $(J,K)$} if
    \begin{equation*}
        P_w^J(t) = P_v^K(t) \cdot P_u^J(t).
    \end{equation*}
    When $J = \emptyset$, we simply say that $w = vu$ is a BP decomposition with
    respect to $K$.
\end{defn}
Some elementary equivalent definitions of BP decompositions are
given in Proposition \ref{P:bplemma}. In particular, checking whether or not a
given parabolic decomposition is a BP decomposition is computationally easy,
and does not require working with Bruhat order.

\begin{example}\label{Ex:BP_A3}
    Let $G=\SL_4(k)$, with Weyl group $W$ generated by $S=\{s_1,s_2,s_3\}$.  If
    $J=\{s_1,s_3\},$ then the parabolic decomposition $w=vu=(s_1s_3s_2)(s_3s_1)$ is
    a BP decomposition with respect to $J$ since
    \begin{align*}
        P^{\emptyset}_w(t)=P^J_v(t)\cdot P^{\emptyset}_u(t)&=(t^3+2t^2+t+1)(t^2+2t+1)\\
        &=t^5+4t^4+6t^3+5t^2+3t+1.
    \end{align*}
    The parabolic decomposition $w=vu=(s_1s_3s_2)(s_1)$ is not a BP decomposition with respect to $J$ since
    \begin{equation*}
        P^{\emptyset}_w(t)=t^4+3t^3+4t^2+3t+1
    \end{equation*}
    and
    \begin{equation*}
        P^J_v(t)\cdot P^{\emptyset}_u(t)=(t^3+2t^2+t+1)(t+1)=t^4+3t^3+5t^2+2t+1.
    \end{equation*}
\end{example}

Our first main theorem is a geometric characterization of BP decompositions.
Note that this theorem holds when $G$ is a general Kac-Moody group.
\begin{thm}\label{T:bp1}
    Let $w \in W^J$ and $w = vu$ be a parabolic decomposition with respect to
    $K$. Then the following are equivalent:

    \begin{enumerate}[(a)]
        \item The decomposition $w = vu$ is a BP decomposition with respect to $(J,K)$.
        \item The projection $\pi:X^J(w) \arr X^K(v)$ is Zariski-locally trivial with fibre $X^J(u)$.
    \end{enumerate}

    Consequently, if $w = vu$ is a BP decomposition then:
    \begin{enumerate}
        \item $X^J(w)$ is (rationally) smooth if and only if $X^J(u)$ and
            $X^K(v)$ are (rationally) smooth.
        \item  The projection $\pi:X^J(w) \arr X^K(v)$ is smooth if and only if
            $X^J(u)$ is smooth.
    \end{enumerate}
\end{thm}

%We note that the second part of Theorem \ref{T:bp1} holds when rationally
%smoothness is used in place of smoothness:
%\begin{prop}[\cite{BP05}]\label{P:ratsmooth}
%    Let $w = vu$ be a right parabolic decomposition of $w$ with respect to
%    $K$, where $w \in W^J$, $J \subset K$. If $X^J(w)$ is rationally smooth,
%    then $X^K(v)$ is rationally smooth. If $w = vu$ is a BP decomposition,
%    then $X^J(w)$ is rationally smooth if and only if $X^J(u)$ and $X^K(v)$
%    are rationally smooth.
%\end{prop}
%When $J = \emptyset$, Proposition \ref{P:ratsmooth} is due to Billey-Postnikov.
%The extension to the relative case is straight-forward, and is explained in
%subsection \ref{SS:combchar}.
%

%\medskip

\begin{example}\label{Ex:forintro}
    Let $G=\SL_4(k)$.  Geometrically,  we have
    \begin{equation*}
        G/B=\{V_\bullet=(V_1\subset V_2\subset V_3\subset k^4)\ |\ \dim V_i=i\}.
    \end{equation*}
    Let $E_\bullet$ denote the flag corresponding to $eB$.  If $w=s_1s_2s_3s_2s_1$, then
    \begin{equation*}
        X^{\emptyset}(w) = \{V_{\bullet}\ |\ \dim(V_2\cap E_2)\geq 1\}.
    \end{equation*}
    If $J=\{s_1,s_3\}$, then $\pi(V_\bullet)=V_2$.  From Example \ref{Ex:BP_A3}, $w=vu=(s_1s_3s_2)(s_3s_1)$ is a BP decomposition with
    respect to $J$ .  In particular, the Schubert variety
    \begin{equation*}
        X^J(v)= \{V_2\ |\ \dim(V_2\cap E_2)\geq 1\}
    \end{equation*}
    and the fibre over $V_2$ in the projection $\pi:X^{\emptyset}(w)\rightarrow X^J(v)$ is
    \begin{equation*}
        \pi^{-1}(V_2)=\{(V_1,V_3)\ |\ V_1\subset V_2\subset V_3\}\iso X^{\emptyset}(u)\iso
        \mathbb{P}^1\times \mathbb{P}^1.
    \end{equation*}
    In this case, the uniform fibre $X^{\emptyset}(u)$ is smooth. However since $X^J(v)$ is
    singular, we have that $X^{\emptyset}(w)$ is singular.  If $J=\{s_1,s_2\}$, then $\pi(V_\bullet)=V_3$ and $w=vu=(s_1s_2s_3)(s_2s_1)$ is not a BP decomposition.  The fiber over $V_3$ is given by
    \begin{align*}
        \pi^{-1}(V_3)&=\{(V_1,V_2)\ |\ V_1\subset V_2\subset V_3
        \text{ and } \dim(V_2\cap E_2)\geq 1\} \\
        & \iso \begin{cases} X^{\emptyset}(s_2s_1) & \text{if}\quad \dim(V_3\cap E_2)=1\\
                             X^{\emptyset}(s_1s_2s_1) & \text{if}\quad E_2 \subset V_3
                \end{cases}
    \end{align*}
    Note that the fibres are not equidimensional.
\end{example}

If $Y$ is a variety over $k$, let $H^*(Y)$ denote either etale cohomology, or,
when $k = \C$, singular cohomology. For etale cohomology, we take coefficients
in $\Q_l$, where $l$ is a prime not equal to the characteristic of $k$, while
for singular cohomology we take coefficients in $\C$.
\begin{cor}\label{C:bpcoh}
    Let $w = vu$ be a parabolic decomposition with respect to $K$. Then the
    following are equivalent:
    \begin{enumerate}[(a)]
        \item The decomposition $w = vu$ is a BP decomposition with respect to $(J,K)$.
        \item There is an isomorphism
            \begin{equation*}
                H^*(X^J(w))\iso H^*(X^K(v)) \otimes H^*(X^J(u))
            \end{equation*}
            as $H^*(X^K(v))$-modules.
    \end{enumerate}
\end{cor}
It is well known that the Poincar\'{e} polynomial $P_w(t^2)=\sum H^i(X^J(w))
t^i$, so the $(b)\Rightarrow (a)$ direction of Corollary \ref{C:bpcoh} follows
immediately from the definition.  The $(a)\Rightarrow (b)$ direction of
Corollary \ref{C:bpcoh} and Theorem \ref{T:bp1} will be proved in Section
\ref{SS:geomchar}.

Our second main theorem concerns the existence of BP decompositions when $G$
is semisimple, or equivalently when $W$ is finite.
\begin{thm}\label{T:bp2}
    Let $w \in W^J$, where $W$ is finite, and suppose $|S(w)
    \setminus J| \geq 2$. If $X^J(w)$ is rationally smooth, then $w$ has a
    Grassmannian BP decomposition with respect to $(J,K)$ for some maximal proper $K$ containing $J.$
\end{thm}
As an application of our main theorems, we get the following extension of the
Ryan-Wolper theorem to arbitrary finite type.

\begin{cor}\label{C:ryanwolper}
    Let $w \in W^J$, where $W$ is finite, and set $m = |S(v) \setminus J|$.
    Then $X^J(w)$ is rationally smooth if and only if there is a sequence
    \begin{equation}\label{E:fibration_seq}
        X^J(w) = X_0 \arr X_1 \arr \cdots \arr X_{m-1} \arr X_m = \Spec k,
    \end{equation}
    where each morphism is a Zariski locally-trivial fibre bundle, and the
    fibres are rationally smooth Grassmannian Schubert varieties.

    \smallskip

    Similarly, $X^J(w)$ is smooth if and only if there is a sequence as in
    \eqref{E:fibration_seq} where all the fibres, or equivalently, all the
    morphisms, are smooth.
\end{cor}
Each projection $X_i\arr X_{i+1}$ in Corollary \ref{C:ryanwolper} corresponds
to a BP decomposition. However, these BP decompositions are not usually
Grassmannian, since the fibre of the projection is Grassmannian rather than
the base.  To deduce Corollary \ref{C:ryanwolper} from Theorem \ref{T:bp2}, we
start with the morphisms $X_i \arr X_{m-1}$ (which do correspond to
Grassmannian BP decompositions), and then apply a certain associativity
property (stated in Lemma \ref{L:bpassoc}) for BP decompositions. Theorem
\ref{T:bp2} and Corollary \ref{C:ryanwolper} are proved in Section
\ref{S:onesided}.

To complete the description of rationally smooth Schubert varieties, we list
all rationally smooth Grassmannian Schubert varieties of finite type.
\begin{thm}\label{T:ratgrass}
    Let $W$ be a finite Weyl group. Suppose $w \in W^J$ for some $J = S
    \setminus \{s\}$, and that $S(w) = S$.  Then $X^J(w)$ is rationally smooth
    if and only if either

    \begin{enumerate}
    \item $w$ is the maximal element of $W^J$, in which case $X^J(w)$ is smooth.
    \item $w$ is one of the following elements:
    \begin{center}
        \begin{tabular}{c|c|l|l|c}
            $W$ & $s$ & $w$ & index set & $X^J(w)$ smooth? \\ \hline
            $B_n$ & $s_1$ & $s_k s_{k+1} \cdots s_n s_{n-1} \cdots s_1$ & $1 < k \leq n$ & no \\
            $B_n$ & $s_k$ & $u_{n,k+1} s_1 \cdots s_k$ & $1 < k < n$ & no \\
            $B_n$ & $s_n$ & $s_1 \ldots s_n$ & $n \geq 2$ & yes \\\hline
            $C_n$ & $s_1$ & $s_k s_{k+1} \cdots s_n s_{n-1} \cdots s_1$ & $1 < k \leq n$ & yes \\
            $C_n$ & $s_k$ & $u_{n,k+1} s_1 \cdots s_k$ & $1 < k < n$ & yes \\
            $C_n$ & $s_n$ & $s_1 \ldots s_n$ & $n \geq 2$ & no \\\hline
            $F_4$ & $s_1$ & $s_4 s_3 s_2 s_1$ & n/a & no \\
            $F_4$ & $s_2$ & $s_3 s_2 s_4 s_3 s_4 s_2 s_3 s_1 s_2$ & n/a & no \\
            $F_4$ & $s_3$ & $s_2 s_3 s_1 s_2 s_1 s_3 s_2 s_4 s_3$ & n/a & yes \\
            $F_4$ & $s_4$ & $s_1 s_2 s_3 s_4$ & n/a & yes \\\hline
            $G_2$ & $s_1$ & $s_2 s_1$, $s_1 s_2 s_1$, $s_2 s_1 s_2 s_1$ & n/a & no \\
            $G_2$ & $s_2$ & $s_1 s_2$ & n/a & yes \\
            $G_2$ & $s_2$ & $s_2 s_1 s_2$, $s_1 s_2 s_1 s_2$ & n/a & no \\
        \end{tabular}
        \end{center}

    The simple generators $\{s_i\}$ are the simple reflections corresponding to the
    labelled Dynkin diagrams in \cite{Bo68}. When $W$ has type
    $B_n$ or $C_n$, we let $u_{n,k}$ be the maximal element in $W^{S
    \setminus \{s_1,s_k\}}\cap W_{S\setminus\{s_1\}}$. In each case, the set $J
    = S \setminus \{s\}$, where $s$ is listed in the table.

    %The first column of this table gives the type of $W$. The second column
    %gives the unique right descent.

    \end{enumerate}
\end{thm}

\begin{rmk}
    All the elements listed in part (2) of Theorem \ref{T:ratgrass} satisfy a
    Coxeter-theoretic property which we term \emph{almost maximality}. This
    property is defined in Definition \ref{D:almostmax}, and plays an important
    role in the proof of Theorem \ref{T:bp2}.
\end{rmk}

\begin{rmk}
    The assumption in Theorem \ref{T:ratgrass} that $S(w)=S$ does not weaken
    the characterization.  If $S(w)$ is a strict subset of $S$, then $X^J(w)$
    is isomorphic to the Schubert variety indexed by $w$ in the smaller flag
    variety $G_{S(w)}/P_{S(w)\cap J}$ corresponding to the algebraic subgroup
    $G_{S(w)} \subset G$ with Weyl group $W_{S(w)}$. For a precise statement,
    see Lemma \ref{L:parabolicsub}.
\end{rmk}

As mentioned in the introduction, the list of smooth Schubert varieties in
a generalized Grassmannian is known in many cases \cite{LW90} \cite{BP99}
\cite{BL00} \cite{Ro12} \cite{HM13}. In particular, Hong and Mok show that if
$X^J(w)$ is a smooth Schubert variety in a generalized Grassmannian
corresponding to a long root, then $w$ must be the maximal element of
$W_{S(w)}^J$ (in the cominuscule case this also follows from the earlier work
of Brion-Polo). The smooth Schubert varieties in $C_n$ with $s = s_k$ with $1 <
k < n$ arise as ``odd symplectic manifolds'', and have been studied by Mihai
\cite{Mi07}. To the best of the authors' knowledge, the cases $F_4$, $s=s_3$
and $s=s_4$ are not covered by previous work, and the completeness of the above
list has not been addressed outside of the cases mentioned above.

It is well known that a Schubert variety $X^J(w)$ is rationally smooth if and
only if the corresponding Kazhdan-Lusztig polynomials are trivial.  While there
are a number of explicit formulas for Kazhdan-Lusztig polynomials of the
Schubert varieties of minuscule and cominuscule generalized Grassmannians (see
sections 9.1 and 9.2 of \cite{BL00} for a summary), the authors' are not aware
of any complete list of rationally smooth Grassmannian Schubert varieties of
finite type in previous work.

Combining Theorem \ref{T:ratgrass} with our main results, we get a new proof of
Peterson's theorem:
\begin{cor}\label{C:peterson}
    Suppose $W$ is simply-laced. Then $X^J(w)$ is rationally smooth if and only
    if it is smooth.
\end{cor}
\begin{proof}
    If $W$ is simply-laced, then all rationally smooth Grassmannian Schubert
    varieties are smooth by Theorem \ref{T:ratgrass}, part (1). So if
    $X^J(w)$ is rationally smooth, then it is smooth by Corollary
    \ref{C:ryanwolper}.
\end{proof}
Corollary \ref{C:peterson} is a consequence of a more general theorem proved by
Peterson, which states that if $W$ is simply-laced then the rationally smooth
and smooth locus of any Schubert variety coincide \cite{CK03}. Peterson's
proof is more algebro-geometric, while our proof is more combinatorial. Unfortunately
our methods do not seem to apply to Peterson's more general theorem.

\section{Characterization of Billey-Postnikov decompositions}\label{S:bpchar}

The goal of this section is to prove Theorem \ref{T:bp1} and Corollary
\ref{C:bpcoh}. We first give several equivalent combinatorial characterizations
of BP decompositions, and then apply these characterizations to the geometry of
Schubert varieties.

\subsection{Combinatorial characterizations}\label{SS:combchar}

In this section, we can assume $W$ is an arbitrary Coxeter group with simple
generating set $S$. Note that Definition \ref{D:BP} still makes sense for
arbitrary Coxeter groups. We restrict to the finite or crystallographic case
only when necessary. We start by proving some important facts about BP
decompositions based on known facts from the case $J = \emptyset$.  For
notational simplicity, define
\begin{equation*}
    W^J_K:=W^J\cap W_K
\end{equation*}
for any $J\subseteq K\subseteq S$.
\begin{lemma}\label{L:maxexist}
    For every element $w \in W^J$ and subset $K \subseteq S$ containing $J$,
    there is a unique maximal element $\bar u$ in $[e,w] \cap W_K^J$ with
    respect to Bruhat order. If $w = vu$ is the parabolic decomposition
    of $w$ with respect to $K$, then $\bar u$ has a reduced decomposition
    $\bar u = \bar v u$, where $\bar v \in [e,v] \cap W_K$.
\end{lemma}
\begin{proof}
    When $J = \emptyset$, the existence of $\bar u$ is proved in \cite[Lemma 7]{vdHo74}. Let
    $u'$ denote the maximal element of $[e,w] \cap W_K$ and let $u' = \bar u u''$ be
    the parabolic decomposition of $u'$ with respect to $J$, so $\bar u \in
    W^J_K$ and $u'' \in W_J$. If $u_0\in [e,w]\cap W^J_K$, then $u_0 \leq u$ and
    hence $u_0 \leq \bar u$.

    For the second part of the lemma, if we take $J = \emptyset$ then it is
    easy to prove by induction on $\ell(u)$ that $u'$ above has a reduced
    decomposition $u' = v' u$, where $v' \in [e,v] \cap W_K$. For arbitrary
    $J$, observe that if $u \in W^J_K$ and $s \in K$ such that $\ell(su)
    = \ell(u) + 1$, then either $su \in W^J_K$, or $su = u t$ for some $t \in J$.
    Indeed, $su = u_1 t$ where $u_1 \in W^J$ and $t \in W_J$. Since $u \leq
    su$, we get that $u \leq u_1$, and if $t \neq e$ then we must have
    $u_1 = u$ and $\ell(t) = 1$. Taking a reduced decomposition $s_{1} \cdots
    s_k$ for $v'$ and considering the products $s_k u$, $s_{k-1} s_k u$, $\ldots$,
    we eventually conclude that $\bar u$ has reduced decomposition $s_{i_1} \cdots s_{i_m} u$,
    where $1 \leq i_1 < \ldots < i_m \leq k$.
\end{proof}
Recall that Bruhat order on $W^J$ induces a relative Bruhat order $\leq_J$ on
the coset space $W / W_J$. By definition,
\begin{equation*}
    w_1 W_J \leq_J w_2 W_J
\end{equation*}
if and only if $\bar w_1 \leq \bar w_2$ in the usual Bruhat order, where $\bar
w_i$ is the minimal length coset representative of $w_i W_J$. Note that if $w_1
\leq w_2$ in Bruhat order, then $w_1 W_J \leq w_2 W_J$ even if $w_1,w_2 \not\in
W^J$. Define the \emph{descent set relative to} $J$ to be
\begin{equation*}
    D^J_L(w):=\{s\in S\ |\ sw W_J \leq_J w W_J\}
\end{equation*}

%Given a subset $J \subset S$, and an element $w \in W^J$, define $D_L^J(w)$ to
%be the set of generators $s \in S$ such that $sw W_J \leq_J w W_J$.

\begin{prop}\label{P:bplemma}
    Let $w = vu\in W^J$ be a parabolic decomposition with respect to $K$, so
    $v \in W^K$, $u \in W_K^J$. The following
    are equivalent:
    \begin{enumerate}[(a)]
        \item $w = vu$ is a BP decomposition with respect to $(J,K)$.
        \item The multiplication map
            \begin{equation*}%\label{E:Mult_map}
                \left([e,v] \cap W^K\right) \times \left([e,u] \cap W^J_K\right)
                    \arr [e,w] \cap W^J
            \end{equation*}
            is surjective.
        \item The element $u$ is the maximal element of $[e,w] \cap W_K^J$.
        \item $S(v) \cap K \subseteq D^J_L(u)$.
    \end{enumerate}
    Furthermore, if $W_J$ is a finite Coxeter group and $u'$ the maximal element
    of coset $u W_J$, then the following are equivalent to parts (a)-(d).

    \begin{enumerate}[(a),start=5]
        \item $S(v) \cap K \subseteq D_L(u')$.
        \item The element $u'$ has reduced decomposition $u_0 u_1$, where
            $u_0$ is the maximal element of $W_{S(v) \cap K}$.
        \item $w = vu'$ is a BP decomposition with respect to $K$.
    \end{enumerate}
\end{prop}
Since the descent and support sets can be calculated efficiently, part (d)
gives a practical criterion for checking whether a parabolic decomposition is a
BP decomposition.
\begin{proof}%[Proof of Proposition \ref{P:bplemma}] We sketch the proof for general $J$.
    Note that the multiplication map in part (b) is always injective.  Hence
    part (b) is equivalent to part (c).  The multiplication map is also length
    preserving, so part (b) is equivalent to part (a).

    To show that parts (c) and (d) are equivalent, suppose $u$ is the maximal
    element of $[e,w] \cap W_K^J$. If $s \in S(v) \cap K$ then $s u \in W_K$
    and hence $s u W_J \leq_J u W_J$. For the converse, suppose $\bar u$ is the
    maximal element of $[e,w] \cap W^J_K$, so $u \leq \bar u \leq w$. By Lemma
    \ref{L:maxexist}, if $u < \bar u$ then there is a simple reflection $s \in
    S(v) \cap K$ such that $u < su \leq \bar u$ and $su \in W^J$. Hence $S(v)
    \cap K$ is not a subset of $D^J_L(u)$.

    The equivalence of parts (e) and (f) is immediate. Let $w'$ be the maximal
    element in the coset $w W_J$. Then $u$ is maximal in $[e,w] \cap W^J_K$ if
    and only if $u'$ is maximal in $[e,w'] \cap W_K$.  This implies part (f) is
    equivalent to part (c).

    Finally, $u'$ is the maximal element of $[e,w'] \cap W_K$ if and only if
    $su'\leq u'$ for all $s\in S(v) \cap K$.  Hence part (g) is equivalent to part
    (e).  This completes the proof.
\end{proof}

In the case when $J=\emptyset$, the equivalence of Proposition
\ref{P:bplemma} parts (a)-(c) is proved in \cite[Theorem 6.4]{BP05},
(a) $\Rightarrow$ (d) in \cite[Lemma 10]{OY10}, and (a) $\Leftarrow$ (d) in
\cite[Lemma 2.2]{RS12}. Using Proposition \ref{P:bplemma}, it is easy
to show that BP decompositions are associative, in the same way that
parabolic decompositions are associative:
\begin{lemma}\label{L:bpassoc}
    Let $I\subseteq J \subseteq K \subseteq S$ and $w \in W^I$.  Write $w = xyz$ where $x \in W^K$, $y \in
    W^J_K$, and $z \in W^I_J$. Then the following are equivalent:
    \begin{enumerate}[(a)]
        \item $x(yz)$ is a BP decomposition with respect to $(I,K)$ and $yz$ is
            a BP decomposition with respect to $(I,J)$.
        \item $(xy)z$ is a BP decomposition with respect to $(I,J)$ and $xy$ is a
            BP decomposition with respect to $(J,K)$.
    \end{enumerate}
\end{lemma}

The last combinatorial property concerns Poincar\'{e} polynomials.
\begin{lemma}\label{L:ratsmooth}
    Let $W$ be a crystallographic Coxeter group and $w\in W.$  Let $w = vu\in
    W^J$ be a parabolic decomposition with respect to $K$.  If $w =vu$ is a BP
    decomposition with respect to $(J,K)$, then $P_w^J(t)$ is palindromic if and only if
    $P_v^K(t)$ and $P_u^J(t)$ are palindromic.
\end{lemma}
\begin{proof}
    Let $P_1(t)$ and $P_2(t)$ be polynomials of degree $d_1$ and $d_2$
    respectively.  Suppose $P_j(t)=\sum_i c^j_i t^i$ has the property that
    $c^j_i\leq c^j_{d_j-i}$ for all $i \leq \lfloor d_j / 2 \rfloor$ and
    $j=1,2$.  Then it is easy to check that $P_1\cdot P_2$ is palindromic if
    and only if $P_1$ and $P_2$ are palindromic. By \cite[Theorem A]{BE09}, the
    relative Poincar\'{e} polynomials $P_w^J$ of elements in crystallographic Coxeter
    groups have this property.
\end{proof}

\begin{rmk}
    For arbitrary Coxeter groups, Lemma \ref{L:ratsmooth} holds in the case
    that $J=\emptyset.$  Indeed, it suffices to prove that if $P_w^J(t)$ is
    palindromic, then $P_u^J(t)$ is palindromic.  This follows from \cite[Lemma
    6.6]{BP05} and \cite[Theorem B]{Ca94} along with the recent result in
    \cite{EW14} that the Kazhdan-Lusztig polynomials of arbitrary Coxeter
    groups have nonnegative coefficients.
\end{rmk}

\subsection{Geometric characterizations}\label{SS:geomchar}
In this section we give some geometric properties of BP decompositions,
finishing with the proof of Theorem \ref{T:bp1}. We return to the assumption
that $W$ is the Weyl group of some Kac-Moody group $G$, and hence is
crystallographic.

For the remainder of the section, we fix $J\subseteq K\subseteq S$ and the
corresponding parabolic subgroups $P_J\subseteq P_K\subseteq G$.  For any $g\in G$,
let $[g]\in G/P_K$ denote the image of $g$ under the projection $G\arr G/P_K$.

\begin{lemma}\label{L:fibre}
    Let $w = vu\in W^J$ be a parabolic decomposition with respect to $K$ and
    recall the projection $\pi:X^J(w) \arr X^K(v)$.  Let $[b_0 v_0]$ be a point of
    $X^K(v)$, where $b_0 \in B$ and $v_0 \in [e,v] \cap W^K$. Then
    \begin{equation*}
        \pi^{-1}([b_0 v_0])=b_0 v_0 \bigcup Bu'P_J/P_J
    \end{equation*}
    where the union is over $u' \in W^J_K$ such that $v_0 u' \leq w$.
\end{lemma}
\begin{proof}
    Using the parabolic decomposition of $W_K$, we get that
    \begin{equation*}
        P_K = B W_K B = \bigcup_{u' \in W_K^J} B u' P_J.
    \end{equation*}
    So the fibre of $G \arr G/P_K$ over $[b_0 v_0]$ is
    \begin{equation*}
        b_0 v_0 P_K = \bigcup_{u' \in W_K^J} b_0 v_0 B u' P_J,
    \end{equation*}
    Since $\ell(v_0 u') = \ell(v_0) + \ell(u')$, we have that $b_0 v_0 B u' P_J$ is a subset
    of $B v_0 u' P_J$.  The image of the set $B v_0 u' P_J$ under the projection $G \arr G/P_J$
    lies outside of $X^J(w)$ unless $v_0 u' \leq w$, in which case the image is
    contained in $X^J(w)$.
\end{proof}

Lemma \ref{L:fibre} yields the following intermediate criterion for BP decompositions.
\begin{prop}\label{P:bpfibre}
    Let $w = vu\in W^J$ be a parabolic decomposition with respect to $K$. Then
    $w = vu$ is a BP decomposition with respect to $(J,K)$ if and only if the fibres of
    the projection $\pi:X^J(w) \arr X^K(v)$ are equidimensional.
\end{prop}
\begin{proof}
    Suppose $w = vu$ is a BP decomposition and take $b_0 \in B$, $v_0 \in [e,v]
    \cap W^K$. Then the fibre $\pi^{-1}([b_0 v_0]) =b_0 v_0 X^J(u)$, since if
    $u' \in [e,w] \cap W^J_K$ then $u' \leq u$ by Proposition \ref{P:bplemma} part (c).

    Conversely, the fibre $\pi^{-1}([e])=X^J(\bar u)$, where $\bar u$ is the
    maximal element of $[e,w] \cap W_K^J$, while the fibre
    $\pi^{-1}([v])=vX^J(u)$. Hence if $\pi:X^J(w) \arr X^K(v)$ is
    equidimensional, then we must have $\ell(u) = \ell(\bar u)$. But $u \leq
    \bar u$, so this implies $u = \bar u$.
\end{proof}

To finish the proof of Theorem \ref{T:bp1}, we need two standard lemmas.
\begin{lemma}\label{L:parabolicsub}
    Let $v \in W^K$ and $I=S(v)$. Let $G_I$ be the reductive subgroup of $P_I$,
    and let $P_{I,I \cap K}:=G_I \cap P_K$ be the parabolic subgroup of $G_I$
    generated by $I \cap K$.  Finally, let $X_I^{I \cap K}(v)\subseteq G_I
    / P_{I,I \cap K}$ be the Schubert variety indexed by $v\in W^K_I$.  Then
    the map $G_I / P_{I,I \cap K} \incl G/P_K$ induces an isomorphism $X_I^{I
    \cap K}(v) \arr X^K(v)$.
\end{lemma}
\begin{proof}
    It suffices to show that the induced map $X_I^{I \cap K}(v) \incl X^K(v)$
    is surjective.  Write $P_I=G_IN_I$ where $N_I$ is the unipotent subgroup of
    $P_I$ and let $B_I=G_I\cap B$ denote the Borel of $G_I.$ If $v'\in W_I$, then
    $N_I$ is stable under conjugation by $v'^{-1}$.  Write $B=B_IN_I$. Then for
    any $v'\in W_I$, the Schubert cell
    \begin{equation*}
        Bv'P_K/P_K=B_IN_Iv'P_K/P_K=B_I(v'v'^{-1})N_Iv'P_K/P_K=B_Iv'P_K/P_K.
    \end{equation*}
    Since $v\in W^K_I\subseteq W_I$, we have that  $X_I^{I \cap K}(v) \incl X^K(v)$ is surjective.
\end{proof}
%\begin{proof}
%    Let $N_I$ be the unipotent subgroup of $P_I$, and let $B_I$ be the Borel
%    $G_I \cap B$ of $G_I$. Using the fact that $N_I$ is an ideal of $P_I$, it is
%    easy to see that $B v' P_J = B_I v' P_J$ for any $v' \in W_I^{I \cap J}$. Thus
%    the map $G_I / P_{I,I \cap J} \arr G / P_I$ is a bijection on the level of open Schubert cells.
%    Since $G_I / P_{I,I \cap J} \incl G/P_J$ is a closed immersion, we can identify
%    $X^{I \cap J}_I(v)$ as the closure of the corresponding open Bruhat cell in $G /
%    P_J$.
%\end{proof}

\begin{lemma}\label{L:schubertaction}
    If $u \in W^J$, then $X^J(u)$ is closed under the action of $P_{D^J_L(u)}$, the parabolic
    subgroup generated by the left descent set $D^J_L(u)$ of $u$ relative to $J$.
\end{lemma}
\begin{proof}
    Let $Z$ be the inverse image of $X^J(u)$ under the projection $G\arr G/P_J$. Then
    \begin{equation*}
        Z = \bigcup_{u' \leq u} B u' BW_JB=\bigcup_{u' \leq u} B u'W_J B.
    \end{equation*}
    If $s \in D^J_L(u)$ and $u'W_J\leq_J uW_J$,  then $su'W_J \leq_J uW_J$. Thus
    \begin{equation*}
        s B u'W_J B\subseteq B su'W_J B \cup B u'W_J B \subseteq Z.
    \end{equation*}
    So $Z$ is closed under $D^J_L(u)$, and therefore $X^J(u)$ is closed under $P_{D^J_L(u)}$.
\end{proof}

\begin{proof}[Proof of Theorem \ref{T:bp1}]
    If $\pi : X^J(w) \arr X^K(v)$ is locally-trivial then the fibres of $\pi$ are
    equidimensional.  Thus $w = vu$ is a BP decomposition with respect to
    $(J,K)$ by Proposition \ref{P:bpfibre}.

    Conversely, suppose that $w = vu$ is a BP decomposition with respect to
    $(J,K)$, and let $I = S(v)$ as in Lemma \ref{L:parabolicsub}.  Recall from
    the proof of Proposition \ref{P:bpfibre} that if $b_0 \in B$, $v_0 \in
    [e,v] \cap W^K$, then the fibre $\pi^{-1}([b_0 v_0])=b_0 v_0 X^J(u)$.
    Suppose $g \in G_I$ maps to $[b_0 v_0]$ in $G/P_K$.  Then we can write $g =
    b_1 v_0 p$, where $b_1 \in B_I$ and $p \in P_{I,I \cap K}$. By Proposition
    \ref{P:bplemma} part (d), $I \cap K \subseteq D^J_L(u)$.  Hence by Lemma
    \ref{L:schubertaction}, $p X^J(u) = X^J(u)$, and we conclude that $g X^J(u)
    = b_1 v_0 X^J(u)$ is the fibre of $\pi$ over $[g] = [b_1 v_0]$.

    By \cite[Corollary 7.4.15 and Exercise 7.4.5]{Ku02} the projection $G_I \arr G_I / P_{I,I\cap K}$ is locally trivial, and
    thus has local sections. Given $x \in X^K(v)$, there is a Zariski open
    neighbourhood $U_x \subseteq X^K(v)$ of $x$ with a local section $s : U_x \arr
    G_I \subseteq G$ of the projection $G \arr G / P_K$. Let $m : U_x \times X^J(u)
    \arr G / P_J$ denote the multiplication map $(u,y) \mapsto s(u)\cdot y$.
    The image of $m$ is contained in $X^J(w)$, and thus we get a commuting square
    \begin{equation*}
        \xymatrix{ U_x \times X^J(u) \ar[r]^{m} \ar[d] & X^J(w) \ar[d]^{\pi} \\
                    U_x\ \ \ar @{^{(}->}[r] & X^K(v) \\ }
    \end{equation*}
    in which the fibres of projection $U_x \times X^J(u) \arr U_x$ are mapped
    bijectively onto the fibres of $\pi$. If $z \in \pi^{-1}(U_x)$ and $g = s(\pi(z))$, then $z \in g X^J(u)$. So $m$ maps bijectively onto $\pi^{-1}(U_x)$, and we can define an inverse $\pi^{-1}(U_x) \arr U_x \times
    X^J(u)$ by $z \mapsto (\pi(z), g^{-1} z)$ where $g = s(\pi(z))$. We conclude that $m$ is
    an isomorphism, and ultimately that $\pi:X^J(w) \arr X^K(v)$ is locally
    trivial.

    Now the projection $U_x \times X^J(u) \arr U_x$ is smooth if and only if
    $X^J(u)$ is smooth, and thus the projection $\pi:X^J(w) \arr X^K(v)$ is smooth
    if and only if $X^J(u)$ is smooth. In particular, if $X^J(u)$ and $X^K(v)$
    are both smooth, then $X^J(w)$ is smooth. Conversely, if $X^J(w)$ is smooth
    then the product $U_x \times X^J(u)$ must be smooth whenever the projection
    $G_I \arr X^K(v)$ has a local section over $U$. Looking at Zariski tangent
    spaces, we conclude that $\dim T_x X^K(v) + \dim T_y X^J(u) \leq \ell(w)$
    for all $x \in X^K(v)$, $y \in X^J(u)$. Since $\ell(w) = \ell(u) +
    \ell(v)$, both $X^K(v)$ and $X^J(u)$ must be smooth.

    The Schubert variety $X^J(w)$ is rationally smooth if and only if $X^J(u)$ and $X^K(v)$
    are rationally smooth by Lemma \ref{L:ratsmooth}.
\end{proof}

We finish the section by proving Corollary \ref{C:bpcoh}.
\begin{proof}[Proof of Corollary \ref{C:bpcoh}]
    For singular cohomology, the proof follows easily from the Leray-Hirsch
    theorem. For etale cohomology, we use the Leray-Serre spectral sequence
    \begin{equation*}
        E_2^{*,*}=  H^r_{et}\left(X^K(v), R^s\pi_* \Q_l\right) \Longrightarrow
            H^{r+s}_{et}\left(X^J(w), \Q_l\right)
    \end{equation*}
    for the projection $\pi : X^J(w) \arr X^K(v)$ (see, e.g. \cite{Ta94}).
    Since $\pi$ is locally trivial, the sheaf $R^* \pi_* \Q_l$ is locally
    constant, and by the proper base change theorem we see that it is in fact
    isomorphic to $H^*_{et}(X^J(u), \Q_l)$. Since
    $H^*_{et}(X^K(v), \Q_l)$ and $H^*_{et}(X^J(u), \Q_l)$ are
    concentrated in even dimensions, the Leray-Serre spectral sequence
    collapses at the $E_2$-term, and the spectral sequence converges to
    $H^*_{et}(X^K(v)) \otimes H^*_{et}(X^J(u))$ as an algebra.  Using the
    action of $H^*_{et}(X^K(v))$ on $H^*_{et}(X^J(w))$, we can
    solve the lifting problem to get an isomorphism
    \begin{equation*}
        H^*_{et}\left(X^K(v),\Q_l\right) \otimes H^*_{et}\left(X^J(u), \Q_l\right)
            \arr H^*_{et}\left(X^J(w), \Q_l\right)
    \end{equation*}
    of $H^*_{et}(X^K(v))$-modules.
\end{proof}

\section{Rationally smooth Grassmannian Schubert varieties}\label{S:bpexist}

In this section we define almost maximal elements of a Weyl group and prove
Theorem \ref{T:ratgrass}.  We take $G$ to a be simple Lie group of finite type,
and hence the Weyl group $W$ is a finite Coxeter group. It is
well known that simple Lie groups are classified into four classical families
$A_n,B_n,C_n,D_n$ and exceptional types $E_6,E_7,E_8,F_4$ and $G_2$. We begin
with the following theorem.
\begin{thm}[\cite{Bi98}, \cite{Ga98}, \cite{La98}, \cite{BP05}, \cite{OY10}]\label{T:leafbp}
    Let $X^{\emptyset}(w)$ be rationally smooth Schubert variety with $|S(w)| \geq 2$. Then there is a
    leaf $s \in S(w)$ of the Dynkin diagram of $W_{S(w)}$ such that either $w$ or
    $w^{-1}$ has a BP decomposition $vu$ with respect to $J = S \setminus \{s\}$.

    Furthermore, $s$ can be chosen so that $v$ is either the maximal length
    element in $W^J$, or one of the following holds:
    \begin{enumerate}[(a)]
        \item $W_{S(v)}$ is of type $B_n$ or $C_n$, with either
        \begin{enumerate}[(1)]
            \item $J = S \setminus \{s_1\}$, and $v = s_ks_{k+2} \cdots s_n s_{n-1} \cdots s_1$, for some
                $1 < k \leq n$.
            \item $J = S \setminus \{s_n\}$ with $n\geq 2,$ and $v = s_1 \cdots s_n$.
        \end{enumerate}
        \item $W_{S(v)}$ is of type $F_4$, with either
        \begin{enumerate}[(1)]
            \item $J = S \setminus \{s_1\}$ and $v = s_4 s_3 s_2 s_1$.
            \item $J = S \setminus \{s_4\}$ and $v = s_1 s_2 s_3 s_4$.
        \end{enumerate}
        \item $W_{S(v)}$ is of type $G_2$, and $v$ is one of the elements
            \begin{equation*}
                s_2s_1,\ s_1s_2s_1,\ s_2s_1s_2s_1,\ s_1s_2,\ s_2s_1s_2,\ s_1s_2s_1s_2.
            \end{equation*}
    \end{enumerate}
\end{thm}
Note that the elements listed in parts (a)-(c) of Theorem \ref{T:leafbp}
correspond to the elements listed in part (2) of Theorem \ref{T:ratgrass} for
which $s$ is a leaf of the Dynkin diagram.  The result that $w$ or $w^{-1}$ has
a BP decomposition with respect to a leaf is due to Billey \cite[Theorem 3.3
and Proposition 6.3]{Bi98} in the classical types and Billey-Postnikov
\cite{BP05} in the exceptional types.  The type $A$ case was also proved by
Gasharov \cite{Ga98} and Lascoux \cite{La98}.  For the second part of Theorem
\ref{T:leafbp} on the presentation of $v$, the proof for the classical types is
again due to Billey \cite{Bi98}.  The type $E$ case is due to Oh-Yoo in
\cite{OY10} \footnote{This also follows from the geometric results of
\cite{HM13} together with Peterson's theorem that all rationally smooth
elements in type $E$ are smooth.} and types $F_4$ and $G_2$ can be easily
verified by computer using Kumar's criteria for rational smoothness
\cite{Ku96}.  We remark that computer calculation plays an important role in
the proof of Theorem \ref{T:leafbp}.  In particular, the results of \cite{BP05}
and \cite{OY10} on the exceptional types both require exhaustive computer
verification.

\begin{comment}
Again in the exceptional
types, if $s$ is a leaf then $W^{S \setminus \{s\}}$ is small enough\footnote{For instance,
in $E_8$ if $s = s_8$ then $|W^{S \setminus \{s\}}| = 240$. If $s=s_3$ then
$|W^{S \setminus\{s\}}| = 17280$.} that the list of rationally smooth elements
in $W^{S \setminus \{s\}}$ can be easily determined by computer.

In particular, Oh and Yoo have shown for type $E$ that only the maximal length elements of
$W^{S \setminus \{s\}}$ are rationally smooth when $s$ is a leaf
\cite{OY10}.\footnote{This also follows from the geometric results of
\cite{HM13} together with Peterson's theorem that all rationally smooth
elements in type $E$ are smooth.}

In type $F_4$, there are two rationally
smooth elements in addition to the maximal elements. Finally, in type $G_2$
every Grassmannian element is rationally smooth.
\end{comment}

Note that the condition that either $w$ or $w^{-1}$ has a BP decomposition in
Theorem \ref{T:leafbp} can be rephrased as $w$ having ``left" or ``right" sided
BP decompositions.  For any $J\subseteq S,$ let ${}^{J}W \simeq W_J\backslash
W$ denote the set of minimal length left sided coset representatives.  Any
$w\in W$ has unique \emph{left sided parabolic decomposition} $w=uv$ with
respect to $J$ where $u\in W_J$ and $v\in {}^{J}W.$  We say a left sided
parabolic decomposition $w=uv$ is a \emph{left sided BP decomposition} with
respect to $J$ if
\begin{equation*}
    P_w(t)=P_u(t)\cdot {}^{J}P_v(t)
\end{equation*}
where
\begin{equation*}
    {}^{J}P_v(t):=\sum_{x\in [e,v]\cap {}^{J}W}t^{\ell(x)}.
\end{equation*}
By a \emph{right sided parabolic or BP decomposition} $w=vu$ with respect to
$J,$ we simply mean a usual parabolic or BP decomposition where $v\in W^J$ and
$u\in W_J$. With this terminology, $w=vu$ is a right sided BP decomposition if
and only if $w^{-1}=u^{-1}v^{-1}$ is a left sided BP decomposition.  The
combinatorial characterizations given in Proposition \ref{P:bplemma} have left
sided BP decomposition analogues.  In particular, for Proposition
\ref{P:bplemma}, part (e) we let $u'$ be the maximal element of the coset
$(W_J)u$ and replace the left descents $D_L(u')$ with right descents $D_R(u')$.

The elements listed in Theorem \ref{T:leafbp} parts (a)-(c) share an important
property we term ``almost maximality". Recall that an element $w \in W$ is the
maximal element in $W_{S(w)}$ if and only if $D_L(w) = S(w)$, or equivalently
if $D_R(w) = S(w)$.  An element $w \in W^J$ is the maximal element in
$W^{S(w)\cap J}_{S(w)}$ if and only if the longest element of $w W_{S(w)\cap
J}$ is in fact the longest element of $W_{S(w)}$. Based on these properties
of maximal elements, we make the following definition.

\begin{defn}\label{D:almostmax}
    Given $w \in W^J$, let $w'$ be the longest element in $w W_{S(w)\cap J}$.
    We say that $w$ is \emph{almost-maximal} in $W^{S(w)\cap J}_{S(w)}$ if all of the following are true.
    \begin{enumerate}[(a)]
        \item There are elements $s,t \in S(w)$ (not necessarily distinct)
            such that
            \begin{equation*}
                D_R(w') = S(w') \setminus \{s\}\quad\text{and}\quad D_L(w')=S(w') \setminus \{t\}.
            \end{equation*}
        \item If $w' = vu$ is the right sided parabolic decomposition with respect to $D_R(w')$, then
            $S(v) = S(w')$.
        \item If $w' = uv$ is the left sided parabolic decomposition with respect to $D_L(w')$, then
            $S(v) = S(w')$.
    \end{enumerate}
    Similarly, if $w\in{}^{J}W$, we say $w$ is \emph{almost-maximal} in
    ${}^{S(w)\cap J}W_{S(w)}$ if parts (a)-(c) are true with $w'$ the longest
    element in the coset $(W_{S(w)\cap J})w$.
\end{defn}
Note that an almost-maximal element is not maximal in $W^{S(w)\cap J}_{S(w)}$
by definition.  If $W^{S(w)\cap J}_{S(w)}$ is clear from context, we will omit it.
By the following lemma, the most interesting case is when $J = S \setminus
\{s\}$ for some $s \in S$.

\begin{lemma}
    Let $w\in W^J$ and assume $J\subseteq S=S(w)$.  If $w$ is almost maximal,
    then there exists $s\notin J$ and parabolic decomposition $w=vu$ with respect
    to $K = S \setminus \{s\}$ such that $v$ is an almost-maximal element of
    $W^K$ and $u$ is the maximal element of $W^J_K$.
\end{lemma}
\begin{proof}
    Assume that $w \in W^J$ is almost-maximal.  Let $w'$ be the longest element of $w
    W_J$.  Let $s$ be the unique element of $S\setminus
    D_R(w')$ and consider the parabolic decomposition $w' = v u'$ with respect
    to $K = S \setminus \{s\}=D_R(w')$.  Note that $S(v)=S(w')=S$ and $u'$ is
    maximal in $W_K$ since $J\subseteq K$. Write $u'=uu_0$ where $u$ and
    $u_0$ are maximal in $W^J_K$ and $W_J$ respectively.  Then
    $$w'=wu_0=v\underbrace{(uu_0)}_{u'}$$ which implies that $w=vu$.  Since $w$
    is almost-maximal and $w'$ is the maximal element of $vW_K$, we have that
    $v$ is almost maximal.
\end{proof}

\subsection{Proof of Theorem \ref{T:ratgrass}}

We begin with the following proposition.

\begin{prop}\label{P:almostmax}
    Let $J= S\setminus\{s\}$ for some $s$ and suppose $v\in W^J$ such that $S(v)=S$ and $X^J(v)$ is rationally smooth.  Then the following are equivalent.
    \begin{enumerate}[(i)]
        \item $v$ is not maximal in $W^J$
        \item $v$ is almost-maximal in $W^J$
        \item $v$ appears on the list in Theorem \ref{T:ratgrass} part (2).
    \end{enumerate}
\end{prop}

\begin{rmk}
If $v\in W^J$ is almost maximal, then $X^J(v)$ is not necessarily rationally smooth.  For example, let $W$ be of type $D_4$ with $J=S\setminus\{s_4\}$.  Then $v=s_1s_3s_2s_4\in W^J$ is almost maximal but $X^J(v)$ is not rationally smooth.
\end{rmk}

Most of Theorem \ref{T:ratgrass} follows from Proposition \ref{P:almostmax},
leaving only the determination of which Schubert varieties listed in Theorem
\ref{T:ratgrass} part (2) are smooth.  Before we prove Proposition
\ref{P:almostmax}, we first analyze the elements arising in parts (a) and (b)
of Theorem \ref{T:leafbp}.

\begin{lemma}\label{L:almostmax1}
    Let $W = B_n$ or $C_n$ and $J = \{s_2,\ldots,s_n\}.$ Let $v = s_k \cdots
    s_{n-1} s_n s_{n-1} \cdots s_1$ where $1 < k \leq n$ and $w'$ be the longest
    element of $vW_J$. Then the following are true:

    \begin{enumerate}
        \item $D_L(w') = S \setminus \{s_{k-1}\}$.
        \item The minimal length representative of $W_{D_L(w')} w'$ is
            $s_{k-1}\cdots s_1 u_{n,k}^{-1}$, where $u_{n,k}$ is the maximal
            length element of $W_J^{J \setminus \{s_k\}}$.
    \end{enumerate}
\end{lemma}
\begin{proof}
    Partition $S$ into $S_1 = \{s_1,\ldots,s_{k-1}\}$ and $S_2 = \{s_k,\ldots, s_n\}$.
    If $w_0$ denotes the maximal element of $W_J$ then $w' = v w_0$.  Since
    the elements of $W_{S_1}$ and $ W_{S_2 \setminus\{s_k\}}$ commute, we can
    write
    \begin{equation*}
        w_0 = u_0 u_1 u_{n,k}^{-1},
    \end{equation*}
    where $u_0$ is maximal in $W_{S_1\cap J}$ and $u_1$ is maximal in $W_{S_2 \setminus \{s_k\}}$.
    Then
    \begin{align*}
        w'  &= (s_k \cdots s_n \cdots s_1)(u_0u_1 u_{n,k}^{-1})\\
            &= (s_k \cdots s_n \cdots s_k u_1) (s_{k-1} \cdots s_1 u_0 ) u_{n,k}^{-1}.
    \end{align*}
    Since $(s_k \cdots s_n \cdots s_k)$ is maximal in $W^{S_2 \setminus \{s_k\}}_{S_2}$,
    we have that $(s_k \cdots s_n \cdots s_k u_1)$ is a maximal element in
    $W_{S_2}$.  In particular, $S_2\subseteq D_L(w').$  Similarly $(s_{k-1}
    \cdots s_1)$ is maximal in $W^{S_1 \setminus \{s_1\}}_{S_1}$, and hence
    $(s_{k-1} \cdots s_1 u_0)$ is maximal in $W_{S_1}$. Consequently
    \begin{equation*}
        s_{k-1} \cdots s_1 u_0 = u_2 s_{k-1} \cdots s_1,
    \end{equation*}
    where $u_2$ is the maximal element in $W_{S_1 \setminus \{s_{k-1}\}}$. Now we have
    \begin{align*}
        w'  &= (s_k \cdots s_n \cdots s_k u_1)\cdot(u_2 s_{k-1} \cdots s_1 u_{n,k}^{-1})\\
            &= (u_2 s_k \cdots s_n \cdots s_k u_1) \cdot (s_{k-1} \cdots s_1 u_{n,k}^{-1}).
    \end{align*}
    Thus $S_1 \setminus \{s_{k-1}\}\subseteq D_L(w')$ and hence
    $S\setminus\{s_{k-1}\}\subseteq D_L(w')$.  Since $w'$ is not maximal in
    $W$, the element $s_{k-1}\notin D_L(w').$ This proves part (1), and part (2)
    follows from the fact that $(u_2 s_k \cdots s_n \cdots s_k u_1)$ is
    maximal in $W_{D_L(w')}$.
\end{proof}
Note that if $k = 2$ in Lemma \ref{L:almostmax1}, then $D_L(w') = J$, and the
minimal length representative of $W_J w'$ is $s_1 u_{n,k}^{-1} = w^{-1}$.

\begin{lemma}\label{L:almostmax2}
    Let $W = B_n$ or $C_n$ and $J = \{s_1,\ldots,s_{n-1}\}$ where $n \geq 2.$  Let $v = s_1 \cdots s_n$
    and $w'$ be the longest element of $v W_J$.  Then the following are true:
    \begin{enumerate}
        \item $D_L(w') = J$
        \item The minimal length representative of $W_{D_L(w')} w'$ is $s_n \cdots s_1$.
    \end{enumerate}
\end{lemma}
\begin{proof}
    If $w_0$ denotes the longest element of $W_J$, then $w' = v w_0$, and we can write
    \begin{equation*}
    w_0 = u_0 s_{n-1} \cdots s_1,
    \end{equation*}
     where $u_0$ is the longest element of $W_{J \setminus \{s_{n-1}\}}$. Hence
    \begin{equation*}
        w' = (s_1 \cdots s_{n-1}s_n)(u_0 s_{n-1}\cdots s_1)=(s_1 \cdots s_{n-1} u_0)(s_n s_{n-1} \cdots s_1).
    \end{equation*}
    The lemma now follows from the fact that $s_1 \cdots s_{n-1} u_0 = w_0$.
\end{proof}

\begin{lemma}\label{L:almostmax3}
    Let $W = F_4$ and $J = \{s_2,s_3,s_4\}.$   Let $v = s_4 s_3 s_2 s_1$ and $w'$ be
    the longest element of $w W_J$. Then the following are true:
    \begin{enumerate}
        \item $D_L(w') = \{s_1, s_3, s_4\}$
        \item The minimal length representative of $W_{D_L(w')} w'$ is $s_2 s_1 s_3 s_2 s_4 s_3 s_4 s_2 s_3$.
    \end{enumerate}
\end{lemma}
\begin{proof}
    Computation in $F_4$ shows that
    \begin{align*}
        w' & = (s_4 s_3 s_2 s_1) (s_4 s_3 s_2 s_3 s_4 s_3 s_2 s_3 s_2) \\
           & = (s_1 s_3 s_4 s_3) (s_2 s_1 s_3 s_2 s_4 s_3 s_4 s_2 s_3).
    \end{align*}
\end{proof}
Note that Lemma \ref{L:almostmax3} also applies to $v = s_1 s_3 s_2 s_4$ in
$F_4$, since $F_4$ has an automorphism sending the simple generator $s_k
\mapsto s_{5-k}$ for $k\leq 4$. (This automorphism is not a diagram
automorphism, and hence is not defined on the root system, but it is defined
for the Coxeter group).

\begin{lemma}\label{L:almostmax0}
    Let $v$ be almost-maximal in $W_{S(v)}^{S(v)\cap J}$ and let $w'$ be the
    longest element in $v W_{S(v)\cap J}$.  Let $w' = u_1 v_1$ be the left sided
    parabolic decomposition of $w'$ with respect to $J':=D_L(w')$.
    Then the following are true:
    \begin{enumerate}
        \item $v_1^{-1}$ is almost-maximal in $W_{S(v_1)}^{J'}$
        \item $X^J(v)$ is rationally smooth if and only $X^{J'}(v_1^{-1})$ is rationally smooth.
    \end{enumerate}
\end{lemma}
\begin{proof}
    Part (1) of the lemma is immediate from Definition \ref{D:almostmax} of
    almost-maximal. For the second part, let $u_0$ be the maximal element of
    $W_{S(v)\cap J}$, so $w' = v u_0$.  By Proposition \ref{P:bplemma}, $w' = v
    u_0$ is a BP decomposition and hence by Lemma \ref{L:ratsmooth}, we have
    $X^J(v)$ is rationally smooth if and only if $X^{\emptyset}(w')$ is
    rationally smooth.  But
    \begin{equation*}
        P_{w'}(t) = P_{(w')^{-1}}(t),
    \end{equation*}
    and so $X^{\emptyset}(w')$ is rationally smooth if and only if
    $X^{\emptyset}((w')^{-1})$ is rationally smooth.  Finally, since $u_1$ is the
    maximal element of $W_{J'}$, we have $w' = u_1v_1$ is a left sided BP
    decomposition.  Thus $X^{\emptyset}((w')^{-1})$ is rationally smooth if and
    only if $X^{J'}(v_1^{-1})$ is rationally smooth.
\end{proof}

\begin{proof}[Proof of Proposition \ref{P:almostmax}]
    Clearly (ii) $\Rightarrow$ (i) in the proposition.  We will show (iii)
    $\Rightarrow$ (ii) and (i) $\Rightarrow$ (iii).  We start with the proof of (iii)
    $\Rightarrow$ (ii). Suppose $v\in W^J$ is an element listed in parts (a)-(c) of
    Theorem \ref{T:leafbp}.  If $v$ is of type $G_2$, then it is easy to see that $v$ is
    almost-maximal.  For types $B,C$ and $F$, it follows from Lemmas
    \ref{L:almostmax1}, \ref{L:almostmax2}, and \ref{L:almostmax3} that all
    elements listed in parts (a)-(c) are almost-maximal. As mentioned
    previously, the elements listed in parts (a)-(c) of Theorem \ref{T:leafbp}
    are precisely the elements listed in the table of Theorem
    \ref{T:ratgrass} for which $s$ is a leaf of the Dynkin diagram. By the second
    parts of Lemmas \ref{L:almostmax1}, \ref{L:almostmax2}, and \ref{L:almostmax3},
    if $v'$ is any another element listed in Theorem \ref{T:ratgrass}, then there
    is a leaf $s$ of the Dynkin diagram, and a rationally smooth element $v \in
    W^{S \setminus \{s\}}$ listed in Theorem \ref{T:leafbp}, such that the longest
    element $w'$ in $v W_{S(v) \cap J}$ has left-sided parabolic decomposition
    $w' = u_1 v_1$, where $v' = v_1^{-1}$. Lemma \ref{L:almostmax0}
    now implies that all the elements listed in Theorem \ref{T:ratgrass} are
    almost-maximal and the corresponding Schubert varieties are rationally smooth.
    This proves (iii) $\Rightarrow$ (ii) of the proposition.

    Now we prove (i) $\Rightarrow$ (iii).  Suppose $v\in W^J$ is not maximal and
    let $w'$ be the longest element of $vW_J$.  Theorem \ref{T:leafbp} implies
    there exists a leaf $s'\in S$ such that either $w'$ or $(w')^{-1}$ has a BP
    decomposition $v'u'$ where $v'$ appears on the list given in parts (a)-(c) of
    Theorem \ref{T:leafbp}.  If $w'=v'u'$, then the fact that $D_R(w')=J$ and $w'$
    is not maximal implies that $s'=s$ and hence, $u'=u$ and $v'=v$.  Now suppose
    that $(w')^{-1}=v'u'$ and let $J'=S\setminus\{s'\}$.  By Proposition
    \ref{P:bplemma}, part (e) we have $S(v')\cap J'=J'\subseteq D_L(u').$  Since
    $w'$ is not maximal, we must have $J'=D_L(u')$ and that $u'$ is the longest
    element of $W_J'$.  Thus $w'$ is almost-maximal.  Lemma \ref{L:almostmax0}
    together with Lemmas \ref{L:almostmax1}, \ref{L:almostmax2}, and
    \ref{L:almostmax3} imply that $v$ is an element listed in Theorem \ref{T:ratgrass}.
\end{proof}

The equivalence of Proposition \ref{P:almostmax} parts (ii) and (iii) gives the
following rephrasing Theorem \ref{T:leafbp}.

\begin{cor}\label{C:leafbp2}
    Let $X^{\emptyset}(w)$ be rationally smooth, where $|S(w)| \geq 2$. Then there is a
    leaf $s \in S(w)$ of the Dynkin diagram of $W_{S(w)}$ such that either $w$ or
    $w^{-1}$ has a BP decomposition $vu$ with respect to $J = S \setminus
    \{s\}$. Furthermore, $s$ can be chosen so that $v$ is either maximal or
    almost-maximal in $W_{S(v)}^{S(v)\cap J}$.
\end{cor}

Corollary \ref{C:leafbp2} plays an important role in the proof of Theorem
\ref{T:bp2} in the next section.

We finish the proof of Theorem \ref{T:ratgrass} by determining which Schubert
varieties listed in Theorem \ref{T:ratgrass}, part (2) are smooth.

\begin{lemma}\label{L:almostmaxsmooth}
    Let $v\in W_{S(v)}^{S(v)\cap J}$ and let $w'$ be the
    longest element in $v W_{S(v)\cap J}$.  Let $w' = u_1 v_1$ be the left sided
    parabolic decomposition of $w'$ with respect to $J':=D_L(w')$.  Then $X^J(v)$ is smooth if and only if $X^{J'}(v_1^{-1})$ is smooth.
\end{lemma}
\begin{proof}
    Let $u_0$ be maximal in $W_S(v)\cap J$, so that $w' = v u_0$,
    and let
    \begin{equation*}
        Z = \bigcup_{x \leq w'} B x B
    \end{equation*}
    be the inverse image of $X^J(v)$ in $G$. Then $Z$ is a principal
    $P_J$-bundle over $X^J(w)$, so $Z$ is smooth if and only if $X^J(w)$ is
    smooth.  But $Z$ is isomorphic to the inverse image
    \begin{equation*}
        \bigcup_{x \leq (w')^{-1}} B x B
    \end{equation*}
    of $X^{J'}(v_1^{-1})$ in $G$, so the lemma follows.
\end{proof}

By Lemma \ref{L:almostmaxsmooth} and Lemmas \ref{L:almostmax1},
\ref{L:almostmax2}, and \ref{L:almostmax3}, it suffices to determine when
$X^J(w)$ is smooth for $J = S \setminus \{s\}$, $s$ a leaf.  Note that Kumar
has given a general criterion for smoothness of Schubert varieties, and by this
criterion the smoothness of Schubert varieties is independent of
characteristic \cite{Ku96}.

If $W$ is of type $B_n$ or $C_n$, then the singular locus of $X^J(w)$ when $J =
S \setminus \{s\}$, $s$ a leaf, is well-known (see pages 138-142 of
\cite{BL00}). In type $G_2$, the smooth Schubert varieties are also well-known
(see the exercise on page 464 of \cite{Ku02}). Finally, for type $F_4$
we use a computer program to apply Kumar's criterion to the Schubert varieties
in question.
\begin{rmk}
    Note that we only use prior results on smoothness for the rationally smooth
    almost-maximal elements, which do not occur in the simply-laced case.
    Hence the proof of Peterson's theorem (Corollary \ref{C:peterson}) depends
    only on Theorem \ref{T:leafbp}.
\end{rmk}

\section{Existence of Billey-Postnikov decompositions}\label{S:onesided}

In this section, we prove Theorem \ref{T:bp2} and Corollary \ref{C:ryanwolper}.
The main theorem, stated below, is an extension of Theorem \ref{T:leafbp},
wherein we show that if $X^{\emptyset}(w)$ is rationally smooth then $w$ has a
right-sided Grassmannian BP decomposition.

\begin{thm}\label{T:onesidebp}
    Let $W$ be a finite Weyl group. Suppose $X^{\emptyset}(w)$ is rationally
    smooth for some $w \in W$. Then one of the following is true:
    \begin{enumerate}[(a)]
        \item The element $w$ is the maximal element of $W_{S(w)}$.
        \item There exists $s \in S(w) \setminus D_R(w)$ such that $w$ has a right sided BP
            decomposition $w = v u$ with respect to $J = S \setminus \{s\}$,
            where $v$ is either the maximal or an almost-maximal element of
            $W^{S(v)\cap J}_{S(v)}$.
    \end{enumerate}
\end{thm}
Note that if $w$ is almost-maximal (relative to $J = \emptyset$) then $w$
satisfies part (b) of Theorem \ref{T:onesidebp} by definition.  The requirement that
$s$ not belong to $D_R(w)$ in part (2) of the theorem is critical both for the
inductive proof of the theorem, and for showing that BP decompositions exist in the
relative case.

We introduce some terminology for subsets of the generating set $S$.
\begin{defn}
    A subset $T \subseteq S$ is \emph{connected} if the Dynkin diagram of $W_T$
    is connected. The \emph{connected components} of a subset $T \subseteq
    S$ are the maximal connected subsets of $T$.
\end{defn}
\begin{defn}
    We say that $s$ and $t$ in $S$ are \emph{adjacent} if $s$ and $t$ are
    adjacent in the Dynkin diagram.  We also say that $s$ is adjacent to a subset $T$ if
    $s$ is adjacent to some $t\in T$, and that two subsets $T_1, T_2$ are adjacent if
    there is some element of $T_1$ adjacent to $T_2$.
    %Elements are always adjacent to themselves.
\end{defn}

We use this terminology for the following lemma:
\begin{lemma}\label{L:paradescent}
    Let $w = vu$ be a right sided parabolic decomposition with respect to some $J
    \subseteq S$. If $s \in S\setminus S(v)$ is adjacent to $S(v)$, then $s \notin D_L(w)$.

    Similarly, let $w = uv$ be a left sided parabolic decomposition with respect to some $J
    \subseteq S$.  If $s \in S\setminus S(v)$ is adjacent to $S(v)$, then $s \notin D_R(w)$.
\end{lemma}
\begin{proof}
    Clearly the second statement of the lemma follows from the first by
    considering $w^{-1}$.  We proceed by induction on the length of $v$. If
    $\ell(v) = 1$ then the proof is obvious. Otherwise take a reduced
    decomposition $v = v_1 t v_0$, where $s$ is adjacent to $t$, but not
    adjacent to $S(v_0)$.  (Note that $v_0$ could be equal to the identity.)
    Then $tv_0 \in W^J$ and since $s \notin S(v)$, we have $s\in D_L(w)$ only
    if $s\in D_L(u)$. Assume that $s \in D_L(u)$ and write
    \begin{equation*}
        w =  (v_1 t v_0)(u)=(v_1 t s v_0) (s u).
    \end{equation*}
    Once again, we have $s\in D_L(w)$ only if  $s\in D_L(tsv_0 (su))$. But
    since $t$ and $s$ are adjacent, we have $s\in D_L(tsv_0 (su))$ only if $t
    \in D_L(v_0 (su))$.  If $t \notin J$, then $t\notin S(v_0 (su))$, and we
    are done.  Otherwise, if $t \in J$, then $t$ must be adjacent to $S(v_0)$ since
    $t v_0 \in W^J$.  But $v_0 \in W^J$, so by induction, we get $t \notin
    D_L(v_0(su))$.
\end{proof}

We now proceed to the proof of Theorem \ref{T:onesidebp}. The proof uses multiple reduced
decompositions for the same element, so we have included certain schematic
diagrams to aid the reader in keeping track of the reduced decompositions under
consideration. For example, if $w=w_1w_2w_3$ is a reduced decomposition with
$v=w_1w_2$ and $u=w_2w_3$, then we diagram this relation by

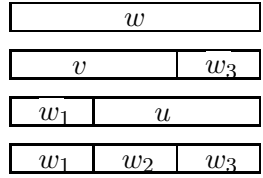
\begin{align*}
    &\text{$\begin{young}[10pt] ]= &==    &==  &== &==w   &==  &== &==    &=] \end{young}$}\\
    &\text{$\begin{young}[10pt] ]= &==    &==v &== &==    &=]  &]= &==w_3 &=] \end{young}$}\\
    &\text{$\begin{young}[10pt] ]= &==w_1 &=]  &]= &==    &==u &== &==    &=] \end{young}$}\\
    &\text{$\begin{young}[10pt] ]= &==w_1 &=]  &]= &==w_2 &=]  &]= &==w_3 &=] \end{young}$}
\end{align*}

\begin{proof}[Proof of Theorem \ref{T:onesidebp}]
    We proceed by induction on $|S(w)|$. It is easy to see that the theorem is
    true if $|S(w)|=1,2$. Hence we can assume that $|S(w)|>2$. We can also
    assume without loss of generality that the Dynkin diagram of $W_{S(w)}$ is
    connected. Since $w$ is rationally smooth, we can apply Corollary
    \ref{C:leafbp2} to get that $w$ has either a right sided BP decomposition $w =
    vu$, or a left sided BP decomposition $w = uv$, with respect to $J = S \setminus
    \{s\}$ where $s \in S$ is a leaf of the Dynkin diagram of $W_{S(w)}$.
    Note that in both cases, $X^{\emptyset}(u)$ is rationally smooth by Lemma
    \ref{L:ratsmooth}.  We now consider four cases, depending on whether we get
    a right or left BP decomposition from Corollary \ref{C:leafbp2}, and
    depending on whether or not $u$ is maximal in $W_{S(u)}$.

    \emph{Case 1}: $w$ has a right sided BP decomposition $w = vu$ as in Corollary
    \ref{C:leafbp2}, where $u$ is maximal in $W_{S(u)}$. If $s\in
    D_R(w),$ then $w$ is the maximal element of $W$. Otherwise, if $s\notin
    D_R(w),$ then the decomposition $w=vu$ satisfies condition (b) in Theorem
    \ref{T:onesidebp}, since $v$ is maximal or almost-maximal in
    $W_{S(v)}^{S(v)\cap J}$ by Corollary \ref{C:leafbp2}.

    \emph{Case 2}: $w$ has a left sided BP decomposition $w = uv$ as in Corollary
    \ref{C:leafbp2}, where $u$ is maximal in $W_{S(u)}$, and $v$ is either
    maximal or almost-maximal in ${}^{S(v)\cap J} W_{S(v)}$. If $S(u)=S(v)\cap J$, then $w$ is either maximal or almost-maximal respectively, in which case we are done.

    Since $S(w)$ is connected, if $S(v)\cap J\subsetneq S(u)$ then we
    can choose $s'\in S(u)\setminus S(v)$ adjacent to $S(v)$.
    Since $u$ is maximal in $W_{S(u)}$, the parabolic decomposition $u=v'u'$
    with respect to $J':=S(u)\setminus\{s'\}$ is a BP decomposition with $v'$ is maximal in $W_{S(u)}^{J'}$. Moreover, $s'\notin D_R(w)$ by Lemma \ref{L:paradescent}. Hence the decomposition $w=v'(u'v)$ satisfies
    condition (b) in Theorem \ref{T:onesidebp}.

    \emph{Case 3}: $w$ has a left sided BP decomposition $w = uv$ as in Corollary
    \ref{C:leafbp2}, where $u$ is not maximal in $W_{S(u)}$. Then by induction,
    we have a right BP decomposition $u=v'u'$ with respect to $J':=S(u)\setminus
    \{s'\}$ for some $s'\in S(u)$, satisfying the conditions of Theorem
    \ref{T:onesidebp}. Since $s' \notin D_R(u)$, we must have $s' \notin
    S(v) \setminus \{s\} \subseteq D_R(u)$, so $w = v' (u'v)$ is a parabolic
    decomposition, and $s' \not\in D_R(w)$. But
    \begin{equation*}
        S(v') \cap J' \subseteq D_L(u')\subseteq D_L(u'v),
    \end{equation*}
    and thus $w = v'(u'v)$ is a BP decomposition with
    respect to $S(w) \setminus \{s'\}$, by Proposition \ref{P:bplemma} (d).
    Since $v'$ is either maximal or almost-maximal by the inductive hypothesis, the
    decomposition $w = v' (u' v)$ satisfies condition (b) of Theorem
    \ref{T:onesidebp}.

    \begin{align*}
        &\text{$\begin{young}[10pt] ]= &==   &==  &== &==w  &== &== &==  &=] \end{young}$}\\
        &\text{$\begin{young}[10pt] ]= &==   &==u &== &==   &=] &]= &==v &=] \end{young}$}\\
        &\text{$\begin{young}[10pt] ]= &==v' &=]  &]= &==u' &=] &]= &==v &=] \end{young}$}
    \end{align*}
    \begin{equation*}
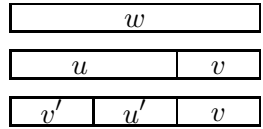
        \text{Figure 1: $w=uv$ in Cases 2 and 3.}
    \end{equation*}

    \emph{Case 4}: $w$ has a right BP decomposition $w = vu$ as in Corollary
    \ref{C:leafbp2}, where $u$ is not maximal in $W_{S(u)}$. Note that if $s
    \notin D_R(w)$ then condition (b) of Theorem \ref{T:onesidebp} is
    satisfied immediately, and we would be done. We consider several subcases where we either prove $s
    \notin D_R(w)$ or we find $s'\notin D_R(w)$ that satisfies condition (b) of Theorem \ref{T:onesidebp}.

    First, suppose that $v$ is almost maximal in $W^{S(v)\cap J}_{S(v)}$. Take
    a reduced decomposition $u = u_0 u_1$, where $u_0$ is the maximal element
    in $W_{S(v) \cap J}$. Then
    \begin{equation*}
        S(v)\cap J\subseteq D_R(v u_0)\subsetneq S(v),
    \end{equation*}
    implying that $s \notin D_R(v u_0)$. It follows
    that $s \not\in D_R(w)$.

    \begin{align*}
        &\text{$\begin{young}[10pt] ]= &==   &==  &== &==w  &== &== &==  &=] \end{young}$}\\
        &\text{$\begin{young}[10pt] ]= &==v  &=]  &]= &==   &== &==u &== &=] \end{young}$}\\
        &\text{$\begin{young}[10pt] ]= &==v &=]  &]= &==u_0 &=] &]= &==u_1 &=] \end{young}$}\\
    \end{align*}
    \begin{equation*}
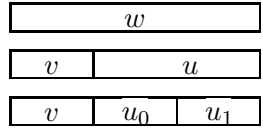
        \text{Figure 2: $w=vu$ with $u$ not maximal and $v$ almost maximal.}
    \end{equation*}

    Now assume $v$ is maximal in $W^{S(v)\cap J}_{S(v)}$, and apply induction to get a
    BP decomposition $u = v' u'$ with respect to $S(u) \setminus \{s'\}$,
    satisfying condition (b) of Theorem \ref{T:onesidebp}. By Proposition
    \ref{P:bplemma}, $S(v) \setminus \{s\}\subseteq D_L(u)$, and hence
    if $t \in S(v) \setminus \{s\}$ is adjacent to $S(v')$, then
    $t\in S(v')$ by Lemma \ref{L:paradescent}. Now since $v$
    is Grassmannian, the set $S(v)$ is connected, and since $s$ is a leaf, the
    set $S(v) \setminus \{s\}$ is also connected. We conclude that if $S(v)
    \setminus \{s\}$ is adjacent to $S(v')$, then $S(v) \setminus
    \{s\}\subseteq S(v')$.  Furthermore $s$ is adjacent to $S(v) \setminus
    \{s\}$, and hence $s$ is adjacent to $S(v')$. Conversely, if $S(v)
    \setminus \{s\}$ is non-empty and $s$ is adjacent to some element of
    $S(v')$, then this element must be contained in $S(v) \setminus \{s\}$,
    since $S(v)$ is connected and $s$ is adjacent to a unique element of
    $S(w)$. We conclude that either $S(v)$ is not adjacent to $S(v')$, or $S(v)
    \setminus\{s\}\subseteq S(v')$ with $s$ adjacent to $S(v')$.

    In the former case, when $S(v)$ is not adjacent to $S(v')$, the elements of
    $S(v)$ pairwise commute with the elements of $S(v')$.  Hence the
    decomposition $w = v' (v u')$ is a BP decomposition with respect to $S(w)
    \setminus \{s'\}$. The element $v'$ is either maximal or almost-maximal in
    $W_{S(v')}^{S(v') \setminus \{s'\}}$ by induction. Since $s' \notin D_R(u)$
    and $s'\in J$, we conclude that $s' \notin D_R(w)$.

    \begin{align*}
        &\text{$\begin{young}[10pt] ]= &==   &==  &== &==w  &== &== &==  &=] \end{young}$}\\
        &\text{$\begin{young}[10pt] ]= &==v  &=]  &]= &==   &== &==u &== &=] \end{young}$}\\
        &\text{$\begin{young}[10pt] ]= &==v &=]  &]= &==v' &=] &]= &==u' &=] \end{young}$}\\
        &\text{$\begin{young}[10pt] ]= &==v' &==  &== &==v &=] &]= &==u' &=] \end{young}$}
    \end{align*}
    \begin{equation*}
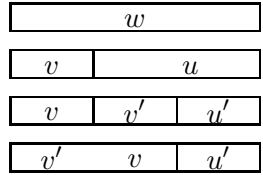
        \text{Figure 3: $w=vu$ with $u$ not maximal, $v$ maximal, and $S(v),S(v')$ not adjacent.}
    \end{equation*}

    This leaves the case that $s$ is adjacent to $S(v')$ and
    $S(v) \setminus \{s\}\subseteq S(v')$. Take a reduced decomposition
    $u' = u_0' u_1'$, where $u_0'$ is maximal in $W_{S(v') \setminus \{s'\}}$.
    Suppose first that $v'$ is maximal, so that $v' u_0'$ is maximal $W_{S(v')}$.
    If $S(v)\setminus\{s\} = S(v')$, then $v v' u_0'$
    is the maximal element of $W_{S(v)}$, and hence there is a BP decomposition
    $v v' u_0' = x_1 x_0$, where $x_1$ is the maximal element of
    $W_{S(v)}^{S(v) \setminus \{s'\}}$ and $x_0$ is the maximal element of
    $W_{S(v) \setminus \{s'\}}$. Since $s' \notin D_R(u)$ and $s' \in J$, we
    have $s' \notin D_R(w)$, and thus $w = x_1 (x_0 u_1')$ is a BP
    decomposition with respect to $S(w) \setminus \{s'\}$ satisfying the
    condition of Theorem \ref{T:onesidebp}. If $S(v) \setminus \{s\}$ is a
    strict subset of $S(v')$, then find $t \in S(v')$ such that $t \notin S(v)
    \setminus \{s\}$.  Consider the parabolic decomposition $v' u_0 = y_0 y_1$,
    where $y_0 \in W_{S(v') \setminus \{t\}}$ and $y_1 \in {}^{S(v') \setminus
    \{t\}} W_{S(v')}$. Since $t \notin S(v)$ and $s$ is adjacent to $S(v') = S(y_1)$, we conclude from Lemma \ref{L:paradescent} that $s \notin D_R(v v' u_0')$. Since $s \notin S(u_1')$, we get that $s \notin
    D_R(w)$.

    \begin{align*}
        &\text{$\begin{young}[13pt] ]= &==  &== &== &==   &==  &==w&== &==   &== &== &=] \end{young}$}\\
        &\text{$\begin{young}[13pt] ]= &== v&=] &]= &==   &==  &== &==u&==   &== &== &=] \end{young}$}\\
        &\text{$\begin{young}[13pt] ]= &==v &=] &]= &==v' &=]  &]= &== &==u' &== &== &=] \end{young}$}\\
        &\text{$\begin{young}[13pt] ]= &==v &=] &]= &==v' &=]  &]= &==u'_0 &=]&]= &== u'_1 &=] \end{young}$}\\
        &\text{$\begin{young}[13pt] ]= &==v &=] &]= &==y_0 &== &==y_1 &== &=] &]= &== u'_1 &=] \end{young}$}\\
        &\text{$\begin{young}[13pt] ]= &==  &== &==x_1&== &==x_0 &== &== &=] &]= &== u'_1 &=] \end{young}$}
    \end{align*}
    \begin{equation*}
        \text{Figure 4: $w=vu$ with $u$ not maximal, $v$ maximal, and $S(v),S(v')$ adjacent.}
    \end{equation*}

    We use a similar argument for the case that $v'$ is almost-maximal in
    $W^{S(v') \setminus \{s\}}_{S(v')}$. By definition $v' u_0'$ is almost maximal in
    $W_{S(v')}$, meaning that $D_L(v' u_0') = S(v') \setminus \{t\}$ for some
    $t \in S(v')$. Since $s' \notin S(u_1)$ and $u_0'$ is maximal in $W_{S(v')
    \setminus \{s'\}}$, the element $u_1'$ belongs to ${}^{S(v')} W_{S(u)}$.
    Thus $t$ does not belong to $D_L(u)$, since otherwise $t \in D_L(v' u_0')$,
    and we conclude that $t \notin S(v)$. Take the unique left sided
    BP decomposition $v' u_0' = y_0 y_1$, where $y_0$ is the maximal element in
    $W_{D_L(v')}$, and $y_1 \in {}^{S(v') \setminus \{t\}} W_{S(v')}$. Then
    $S(y_1) = S(v')$ by the definition of almost-maximal, and since $s$ is
    adjacent to $S(y_1)$, we conclude from Lemma \ref{L:paradescent} that $s \notin D_R(v v' u_0')$. Consequently $s
    \notin D_R(w)$.
\end{proof}

Using Theorem \ref{T:onesidebp}, we can prove Theorem \ref{T:bp2} for relative Schubert varieties.
\begin{proof}[Proof of Theorem \ref{T:bp2}]
    Suppose $X^J(w)$ is rationally smooth and that $|S(w)
    \setminus J| \geq 2$. Let $w_0$ be maximal in $W_J$, so $w' = w w_0$
    is the longest element of $w W_J$. Then $w'$ is rationally smooth by Lemma
    \ref{L:ratsmooth}, so we can apply Theorem \ref{T:onesidebp}. If $w'$ is
    maximal, then $w$ is the maximal element in $W^{J \cap S(w)}_{S(w)}$, and
    hence by choosing any $s \in S(w)\setminus J$ we get a Grassmannian BP
    decomposition of $w$ with respect to $K = S \setminus \{s\}$ as required.

    If $w'$ is not maximal, then there exists $s \in S(w) \setminus D_R(w)$
    such that $w'$ has a BP decomposition $w' = v u$ with respect to $K = S
    \setminus \{s\}$. Since $s \notin D_R(w)$, $s$ must not be in $J$, so
    $u$ has parabolic decomposition $u = u_1 w_0$ with respect to $W_J$, and
    $w$ has parabolic decomposition $w = v u_1$ with respect to $K$. This
    latter decomposition is a BP decomposition by Proposition \ref{P:bplemma}.
\end{proof}

\subsection{The Ryan-Wolper theorem}\label{S:applications}
In this section we use our results to prove Corollary \ref{C:ryanwolper} on
iterated fibre bundle structures of smooth and rationally smooth Schubert
varieties.  We assume that $X^J(w)$ is a rationally smooth Schubert variety of
finite type.

\begin{proof}[Proof of Corollary \ref{C:ryanwolper}]
    Suppose that $W$ is a finite Weyl group, and that $X^J(w)$ is rationally
    smooth. By repeatedly applying Theorems \ref{T:bp1} and \ref{T:bp2}, we can write
    \begin{equation*}
        w = v_m \cdots v_1,
    \end{equation*}
    where $v_i \in W^{J_{i-1}}_{J_{i}}$ for $J_i := J \cup \bigcup_{j\leq i} S(v_j)$, and
    $v_i (v_{i-1} \cdots v_1)$ is a Grassmannian BP decomposition. Let $w_i :=
    v_m \cdots v_{i+1} \in W^{J_i}$, so $w_0 = w$ and $w_m$ is the identity . By Lemma
    \ref{L:bpassoc}, $w_i = w_{i+1} v_{i+1}$ is a BP decomposition with respect to
    $(J_i,J_{i+1})$, so the morphism
    \begin{equation}\label{E:fibre_map}
        X^{J_i}(w_i) \arr X^{J_{i+1}}(w_{i+1})
    \end{equation}
    is a locally-trivial fibre bundle with fibre $X^{J_i}(v_{i+1})$. Hence the
    sequence
    \begin{equation*}
        X^J(w) = X^{J_0}(w_0) \arr X^{J_1}(w_1) \arr \cdots \arr
            X^{J_{m-1}}(w_{m-1}) = X^{J_{m-1}}(v_m) \arr \Spec k
    \end{equation*}
    meets the required conditions in Corollary \ref{C:ryanwolper}. If $X^J(w)$ is
    smooth, then by Theorem \ref{T:bp1} all the fibres $X^{J_i}(v_i)$ are smooth,
    and hence so are the morphisms.

    Conversely, given a locally-trivial morphism $X \arr Y$ with fibre $F$ such
    that the cohomology $H^*(Y)$ of the base and the cohomology $H^*(F)$ of the
    fibre are concentrated in even degrees, we can argue as in the proof of
    Corollary \ref{C:bpcoh} that $H^*(X)$ is isomorphic to $H^*(Y) \otimes H^*(F)$.
    Thus $H^*(X)$ will be concentrated in even degrees, and if $H^*(Y)$ and
    $H^*(F)$ satisfy Poincar\'{e} duality, then so does $H^*(X)$. If there is a
    sequence
    \begin{equation*}
        X^J(w) = X_0 \arr X_1 \arr \cdots \arr X_m = \Spec k
    \end{equation*}
    in which all the morphisms are locally trivial, and all the fibres are
    rationally smooth Schubert varieties, then $H^*(X_i)$ satisfies Poincar\'{e}
    duality for all $i=0,\ldots,m$. In particular, $X^J(w)$ will be rationally
    smooth by the Carrell-Peterson theorem \cite{Ca94}.
\end{proof}
Only limited results are known about BP decompositions outside of finite type.
Billey and Crites have shown that if $X(w)$ is a rationally smooth Schubert
variety of affine type $\tilde{A}$, then (a la Theorem \ref{T:leafbp}) either
$w$ or $w^{-1}$ has a BP decomposition \cite{BC12}. Using this result, the
proof of Theorem \ref{T:onesidebp} extends to the affine setting with minor
modifications, and thus the Ryan-Wolper theorem also holds in affine type
$\tilde{A}$. The framework of BP decompositions also allows us to prove
\cite[Conjecture 1]{BC12}, that $X(w)$ is smooth in affine type $\tilde{A}$ if
and only if $w$ avoids (as an affine permutation) $3412$ and $4231$.

In \cite{RS12}, the authors show that right-sided BP decompositions exist for
rationally smooth Schubert varieties in the full flag varieties of a large
class of non-finite Weyl groups. Hence the Ryan-Wolper theorem also holds for
Schubert varieties $X^{\emptyset}(w)$ in this class via the application of
Theorem \ref{T:bp1}. However, with the exception of $\tilde{A}_3$, all of the
Coxeter groups in this class are of indefinite type. It is an open problem to
prove the existence of BP decompositions for this class when $J$ is non-empty.

Based on this evidence, the following conjecture seems plausible:
\begin{conjecture}
    If $W$ is any Coxeter group, and $w$ belongs to $W^J$ with $P_w^J(t)$
    palindromic, then $w$ has a Grassmannian BP decomposition. As a
    result, the Ryan-Wolper theorem holds in any Kac-Moody flag
    variety.
\end{conjecture}

\bibliographystyle{amsalpha}
\bibliography{palindromic}

\end{document}